\documentclass[10pt]{amsart}
\usepackage{hyperref}

\theoremstyle{plain}
\newtheorem{theorem}{Theorem}
\newtheorem{claim}[theorem]{Claim}
\newtheorem{lemma}[theorem]{Lemma}
\newtheorem{conjecture}[theorem]{Conjecture}
\newtheorem{proposition}[theorem]{Proposition}

\theoremstyle{definition}

\newcommand{\epsi}{\varepsilon}
\newcommand{\fhi}{\varphi}
\newcommand{\norm}[1]{\lVert#1\rVert}
\newcommand{\card}[1]{\lvert#1\rvert}
\newcommand{\conv}[1]{\operatorname{conv}(#1)}
\newcommand{\vol}[1]{\operatorname{vol}(#1)}
\newcommand{\numbersystem}[1]{\mathbf{#1}}
\newcommand{\R}{\numbersystem{R}}
\newcommand{\N}{\numbersystem{N}}

\begin{document}

\bibliographystyle{amsplain}

\title{Large convexly independent subsets of Minkowski sums}
\author{Konrad J.\ Swanepoel}
\thanks{Swanepoel gratefully acknowledges the hospitality of the Department of Applied Mathematics, Charles University, Prague.}
\address{Department of Mathematics,
	London School of Economics and Political Science,
	WC2A 2AE London, UK}
\author{Pavel Valtr}
\thanks{}
\address{Department of Applied Mathematics and Institute for Theoretical Computer Science\\ Charles University \\ Malostransk\'e n\'am.~25. 118 00 Praha 1 \\ Czech Republic}

\begin{abstract}
Let $E_d(n)$ be the maximum number of pairs that can be selected from a set of $n$ points in $\R^d$ such that the midpoints of these pairs are convexly independent.
We show that $E_2(n)\geq \Omega(n\sqrt{\log n})$, which answers a question of Eisenbrand, Pach, Rothvo\ss, and Sopher (2008) on large convexly independent subsets in Minkowski sums of finite planar sets, as well as a question of Halman, Onn, and Rothblum (2007).
We also show that $\lfloor\frac{1}{3}n^2\rfloor\leq E_3(n)\leq \frac{3}{8}n^2+O(n^{3/2})$. 

Let $W_d(n)$ be the maximum number of pairwise nonparallel unit distance pairs in a set of $n$ points in some $d$-dimensional strictly convex normed space.
We show that $W_2(n)=\Theta(E_2(n))$ and for $d\geq 3$ that $W_d(n)\sim\frac12\left(1-\frac{1}{a(d)}\right)n^2$, where $a(d)\in\N$ is related to strictly antipodal families.
In fact we show that the same asymptotics hold without the requirement that the unit distance pairs form pairwise nonparallel segments, and also if diameter pairs are considered instead of unit distance pairs.
\end{abstract}

\maketitle

\section{Three related quantities}
A geometric graph is a graph with the set of vertices in $\R^d$ and with each edge represented as a straight line segment between its incident vertices.
Halman et al.\ \cite{Halman} studied geometric graphs for which the set of midpoints of the edges are \emph{convexly independent}, i.e., they form the vertex set of their convex hull.
For any finite set $P\subset\R^d$ let $E(P)$ be the maximum number of pairs of points from $P$ such that the midpoints of these pairs are convexly independent, and define $E_d(n)=\max_{P\subset\R^d,\card{P}=n} E(P)$.
Halman et al.\ \cite{Halman} asked whether $E_2(n)$ is linear or quadratic.

Motivated by the above question, Eisenbrand et al.\ \cite{Eisenbrand} studied a more general quantity: the maximum size $M_d(m,n)$ of a convexly independent subset of $P+Q$, where $P$ is a set of $m$ points and $Q$ a set of $n$ points in $\R^d$, with the maximum again taken over all such $P$ and $Q$.
(The sets $P$ and $Q$ are not required to be disjoint, but may clearly without loss of generality be assumed to be.)
They showed that $M_2(m,n)=O(m^{2/3}n^{2/3}+m+n)$, from which follows $E_2(n)\leq M_2(n,n)=O(n^{4/3})$, since the midpoints of pairs of points in $P$ are contained in $\frac{1}{2}(P+P)$.
In fact, it holds more generally that $E_d(n)\leq M_d(n,n)$.
They mentioned that they do not know any superlinear lower bound for $M_2(m,n)$.

We now introduce $W_d(n)$ as the maximum number of pairwise nonparallel segments of unit length among a set of $n$ points in some strictly convex $d$-dimensional normed space.
Here the maximum is taken over all sets of $n$ points in $\R^d$ and all strictly convex norms on $\R^d$.
Then it is immediate that $2W_d(n)\leq M_d(n,n)$, since if $P$ has $W$ pairwise nonparallel unit distance pairs in some strictly convex norm with unit sphere $S$, then $P+(-P)$ intersects $S$ in at least $2W$ points.

\section{Asymptotic equivalence}
We now observe that the three quantities $E_d(n)$, $M_d(n,n)$ and $W_d(n)$ are in fact asymptotically equivalent.
Here we consider two functions $f,g:\N\to\N$ to be \emph{asymptotically equivalent} if there exist $c_1,c_2>0$ such that $c_1 f(n)\leq g(n)\leq c_2 f(n)$ for all $n\geq 2$.
We have already mentioned the bounds $E_d(n)\leq M_d(n,n)$ and $2W_d(n)\leq M_d(n,n)$.
\begin{claim}
 \[M_d(n,n)\leq E_d(2n). \]
\end{claim}
\begin{proof}
Let $P$ and $Q$ each be a set of $n$ points such that $P+Q$ contains $M_d(n,n)$ convexly independent points.
Without loss of generality, $P$ and $Q$ are disjoint.
Then $P\cup Q$ is a set of $2n$ points such that the set of midpoints of pairs between $P$ and $Q$ equals $\frac{1}{2}(P+Q)$.
\end{proof}
\begin{claim}
\[ M_d(n,n)\leq 2W_d(2n). \]
\end{claim}
\begin{proof}
Again let $P$ and $Q$ be disjoint sets of $n$ points each such that $P+Q$ contains a convexly independent subset $S$ of size at least $M_d(n,n)$.
There exists a strictly convex hypersurface $C$ symmetric with respect to the origin such that some translate of it contains at least $M_d(n,n)/2$ points from $S$.
Then $P\cup Q$ has at least $M_d(n,n)/2$ pairwise nonparallel unit distances in the norm which has $C$ as unit sphere.
\end{proof}
\begin{claim}
\[M_d(2n,2n)\leq 4M_d(n,n). \]
\end{claim}
\begin{proof}
Let $P$ and $Q$ be two sets of $2n$ points each such that $P+Q$ contains a set $C$ consisting of $M_d(2n,2n)$ convexly independent points.
Let $P=P_1\cup P_2$ and $Q=Q_1\cup Q_2$ be arbitrary partitions such that $\card{P_1}=\card{P_2}=\card{Q_1}=\card{Q_2}=n$.
Label each $p+q\in C$ by $(i,j)$ if $p\in P_i$ and $q\in Q_j$.
Each point in $C$ gets one of the four labels $(1,1)$, $(1,2)$, $(2,1)$, $(2,2)$.
By the pigeon-hole principle, at least $M_d(2n,2n)/4$ points in $C$ have the same label $(i,j)$, which means that they are contained in $P_i+Q_j$.
It follows that $M_d(2n,2n)/4\leq M_d(n,n)$.
\end{proof}
The above claims imply the following.
\begin{proposition}\label{prop}
For any fixed dimension $d$, $M_d(n,n)$, $E_d(n)$, and $W_d(n)$ are asymptotically equivalent.
\end{proposition}

\section{The plane}
The fact that $M_2(n,n)=O(n^{4/3})$ \cite{Eisenbrand} gives Proposition~\ref{prop} nontrivial content in the case $d=2$.
To show that the quantities $E_2(n)$, $M_2(n,n)$, and $W_2(n)$ grow superlinearly, it is sufficient to consider the following smaller quantities.
Let $E_\circ(n)$ denote the largest number of pairs of a set of $n$ points in the Euclidean plane such that the midpoints of these pairs are concyclic (i.e., they lie on the same Euclidean circle).
Let $W_\circ(n)$ denote the largest number of pairwise nonparallel unit distance pairs in a set of $n$ points in the Euclidean plane.
Then clearly $E_2(n)\geq E_\circ(n)$ and $W_2(n)\geq W_\circ(n)$.
As observed in the book of Bra\ss, Moser, and Pach \cite{BMP}, a planar version of an argument of Erd\H{o}s, Hickerson, and Pach \cite{EHP} already gives a superlinear lower bound $W_\circ(n)=\Omega(n\log^\ast n)$.
Here $\log^\ast n$ denotes the iterated logarithm.
In an earlier paper \cite{SV} we showed $W_\circ(n)=\Omega(n\sqrt{\log n})$.
This gives the following.
\begin{theorem}
$E_2(n)$, $M_2(n,n)$, and $W_2(n)$ are all in $\Omega(n\sqrt{\log n})$.
\end{theorem}
Recently it was shown by 
Buchin, Fulek, Kiyomi, Okamoto, Tanigawa, and Cs.\ T\'oth \cite{BFKOTT}
and also by Ond\v{r}ej B\'{\i}lka (personal communication)
that $M_2(m,n)=\Theta(m^{2/3}n^{2/3}+m+n)$.
This implies that $E_2(n)$, $M_2(n,n)$, and $W_2(n)$ are all in $\Theta(n^{4/3})$.

\section{Higher dimensions}
When $d\geq 3$, Proposition~\ref{prop} has empty content, since then the functions $E_d(n)$, $M_d(n,n)$, and $W_d(n)$ are all in $\Theta(n^2)$, since, as shown by Halman et al.\ \cite{Halman}, $M_d(m,n)=mn$ for all $d\geq 3$.
They also showed that $E_d(n)=\binom{n}{2}$ for $d\geq 4$, which leaves only the $3$-dimensional case of this function.

\subsection{Convexly independent subsets of Minkowski sums in $3$-space}

\begin{theorem}\label{thmE}
$\lfloor\frac{1}{3}n^2\rfloor\leq E_3(n)\leq \frac{3}{8}n^2+O(n^{3/2})$.
\end{theorem}
\begin{proof}
For the lower bound it is sufficient to construct, for each natural number $k$, three collections $B_1,B_2,B_3$ of $k$ points each in $\R^3$ such that $\frac12(B_1+B_2)\cup\frac12(B_2+B_3)\cup\frac12(B_3+B_1)$ is convexly independent.
In fact we will construct three infinite collections with this property.

Consider a cube with side length $2$ and center $o$.
Let $I_1$, $I_2$, $I_3$ be three of its edges with a common vertex.
If, for each $i=1,2,3$, we let $A_i$ be a small subinterval of $I_i$ such that $A_i$ and $I_i$ have the same midpoint, then for each triple $i,j,k$ with $\{i,j,k\}=\{1,2,3\}$, $\frac12(A_i+A_j)$ is a small rectangle in the plane $\Pi_k$ through $I_i$ and $I_j$.
Then the set $\bigcup_{i<j}\frac12(A_i+A_j)$ is in convex position, in the sense that each of its points is on the boundary of its convex hull.
It is not convexly independent, however.
Note that $\frac12(A_i+A_k)$ and $\frac12(A_j+A_k)$ are both a distance of almost $1/2$ from $\Pi_k$ and are in the same open half space as $o$.

Now we replace each $A_i$ by a sufficiently small strictly convex curve $B_i$, arbitrarily close to $A_i$, in the plane $\Sigma_i$ through $o$ and $I_i$, curved in such a way that $B_i\cup\{o\}$ is in strictly convex position.
For example, we may take $B_i$ to be a small arc of a circle with center $o$ and radius $\sqrt{2}$, around the midpoint of $I_i$.

At each point $p$ of $B_i$ there is a line $\ell_p$ supporting $B_i$ at $p$ in the plane $\Sigma_i$.
For each plane $\Pi$ through $\ell_p$ except $\Sigma_i$, $B_i\setminus\{p\}$ and $o$ lie in the same open half space bounded by $\Pi$.
Note that $\ell_p$ is almost parallel to $I_i$, because $B_i$ is close to $A_i$.

Now let $\{i,j,k\}=\{1,2,3\}$ and consider points $p\in B_i$, $q\in B_j$, and let $\ell_p$ and $\ell_q$ be as above.
Let $\Sigma$ be the plane through $o$ containing lines parallel to $\ell_p$ and $\ell_q$.
Then by the previous paragraph, $p+\Sigma$ is a plane supporting $B_i$ at $p$ such that $B_i\setminus\{p\}$ lies in the same open half space as $o$, with a similar statement for $q+\Sigma$.
It follows that $\frac12(p+q)+\Sigma$ is a plane supporting $\frac12(B_i+B_j)$ at $\frac12(p+q)$ such that $\frac12(B_i+B_j)\setminus\{\frac12(p+q)\}$ lies in the same open half space as $o$.
Since $\ell_p$ is almost parallel to $I_i$ and $\ell_q$ almost parallel to $I_j$, $\Sigma$ is almost parallel to $\Pi_k$ (the plane through $I_i\cup I_j$).
Thus $\frac12(p+q)+\Sigma$ is a small perturbation of $\Pi_k$.
Since $\frac12(B_i+B_k)$ and $\frac12(B_j+B_k)$ are at a distance of almost $1/2$ from $\Pi_k$, they will also be in the same open half space  determined by $\frac12(p+q)+\Sigma$ as $o$.
It follows that $\bigcup_{i<j}\frac12(B_i+B_j)\setminus\{\frac12(p+q)\}$ is in an open half space bounded by $\frac12(p+q)+\Sigma$.

It follows that $\bigcup_{i<j}\frac12(B_i+B_j)$ is in strictly convex position.
We may now choose $k$ points from each $B_i$ to find a set of $3k$ points in $\R^3$ with the midpoints of $3k^2$ pairs of points in strictly convex position.

For the upper bound it follows from refinements of the Erd\H{o}s-Stone theorem (see e.g.\ \cite{ES}) that it is sufficient to show that any geometric graph such that the midpoints of the edges are convexly independent, does not contain $K_{2,2,2,2,2}$, the complete $5$-partite graph with two vertices in each class.

Thus assume for the sake of contradiction that there exist five sets $C_i$, $i=1,2,3,4,5$, of two points each in $\R^3$, such that $\bigcup_{i<j}\frac12(C_i+C_j)$ is convexly independent.
In particular, if we choose a $c_i\in C_i$ for each $i$, we obtain that the $10$ midpoints of $\{c_1,\dots,c_5\}$ are convexly independent.
As proved by Halman et al.\ \cite{Halman}, the set $\{c_1,\dots,c_5\}$ cannot then itself be convexly independent.
On the other hand, the union of any $4$ of the $C_i$s must be convexly independent.
Indeed, for any fixed $c_1\in C_1$, since $\frac12(c_1+\bigcup_{j=2}^5 C_j)$ must be convexly independent, the union $\bigcup_{j=2}^5 C_j$ is also convexly independent.
Now choose $4$ points from different $C_i$s such that their convex hull has largest volume among all such choices.
Without loss of generality, we may assume that these points are $c_i\in C_i$, $i=1,2,3,4$.
For any $c_5\in C_5$, as mentioned above, the set $\{c_1,\dots,c_5\}$ is not convexly independent, i.e., one of the points is in the convex hull of the others.
If e.g.\ $c_1$ is in the convex hull of $c_2c_3c_4c_5$, then $c_2c_3c_4c_5$ has larger volume, a contradiction.
Similarly, none of $c_2$, $c_3$, $c_4$ can be in the convex hull of the other four.
Thus $c_5$ must be in the convex hull of $c_1c_2c_3c_4$.
Similarly, the other point $c_5'\in C_5$ is also in the tetrahedron $c_1c_2c_3c_4$.
The ray from $c_5$ through $c_5'$ intersects one of the faces of this tetrahedron, say the triangle $c_1c_2c_3$.
Then $\{c_1,c_2,c_3,c_5,c_5'\}$ is not convexly independent.
It follows that $C_1\cup C_2\cup C_3\cup C_5$ is not convexly independent, which contradicts what we have already shown.
\end{proof}

Note that by the Erd\H{o}s-Stone theorem, one of the two bounds in Theorem~\ref{thmE} must be asymptotically correct.
Indeed, either there is some upper bound to $c\in\N$ for which the complete $4$-partite graph $K_{c,c,c,c}$ is realizable, from which the Erd\H{o}s-Stone theorem gives $E(n)\leq n^2/3 +o(n^2)$, or there is no such upper bound, which trivially gives the lower bound $3n^2/8$.
We conjecture that $K_{c,c,c,c}$ is not realizable for some $c\in\N$.
It would be sufficient to prove the following.

\begin{conjecture}
For some $\epsi>0$ the following holds.
Let $A_i=\{p_i,q_i\}$, $i=1,2,3,4$, be four sets of two points each in $\R^3$, such that $\norm{p_i-q_i}_2<\epsi$.
Then the set of midpoints between different $A_i$,
\[\bigcup_{\substack{i,j=1,2,3,4,\\ i\neq j}}\frac12(A_i+A_j),\] is not convexly independent.
\end{conjecture}

\subsection{Pairwise nonparallel unit distance pairs in strictly convex norms}

The function $W_d(n)$ is related to large strictly antipodal families, as studied by Martini and Makai \cite{MM, MM2} and others \cite{CKSW}.
We introduce the following related quantities.

Let $U_d(n)$ be the largest number of unit distance pairs that can occur in a set of $n$ points in a strictly convex $d$-dimensional normed space.
Let $D_d(n)$ be the largest number of diameter pairs that can occur in a set of $n$ points in a strictly convex $d$-dimensional normed space, where a diameter pair is a pair of points from the set whose distance equals the diameter of the set (in the norm).
As in the definition of $W_d(n)$, for both $U_d(n)$ and $D_d(n)$ we take the maximum over all sets of $n$ points in $\R^d$ and all strictly convex norms on $\R^d$.
Then clearly $W_d(n)\leq U_d(n)$ and $D_d(n)\leq U_d(n)$.
Our final result is the observation that these three functions are in fact asymptotically equal for each $d\geq 3$.
To this end we use the notion of a strictly antipodal family of sets.
Let $\{A_i\colon i\in I\}$ be a family of sets of points in $\R^d$.
We say that this family is \emph{strictly antipodal} if for any $i,j\in I$, $i\neq j$, and any $p\in A_i$,  $q\in A_j$, there is a linear functional $\fhi:\R^d\to \R$ such that $\fhi(p)< \fhi(r)< \fhi(q)$ for any $r\in \bigcup_{i\in I}A_i\setminus\{p,q\}$.
Let $a(d)$ denote the largest $k$ such that for each $m$ there exists a strictly antipodal family of $k$ sets in $\R^d$, each of size at least $m$.
It is known that $c^d<a(d)<2^d$ for some $c>1$, and $3\leq a(3)\leq 5$ \cite{MM}.

\begin{theorem}
\[\lim_{n\to\infty}\frac{W_d(n)}{n^2}=\lim_{n\to\infty}\frac{U_d(n)}{n^2}=\lim_{n\to\infty}\frac{D_d(n)}{n^2}=\frac12\left(1-\frac{1}{a(d)}\right).\]
\end{theorem}

\begin{proof}
Suppose first $\{A_i\colon i=1,\dots,a(d)\}$ is a strictly antipodal family of sets in $\R^d$, each of size $k$, where $k\in\N$ is arbitrary.
We may perturb these points such that the family remains strictly antipodal, so that no two segments between pairs of points from $\bigcup_i A_i$ are parallel.
It follows from the definition of strict antipodality that $\bigcup_{i,j,i\neq j}(A_i-A_j)$ is a centrally symmetric, convexly independent set of points.
There exists a centrally symmetric, strictly convex surface $S$ through these points.
The set $S$ defines a strictly convex norm on $\R^d$ such that the distance between any two points in different $A_i$ is a unit distance.
Note that all distances between points in $\bigcup_i A_i$ are at most $1$.
This gives two lower bounds 
\[W_d(n),D_d(n)\geq \frac12\left(1-\frac{1}{a(d)}\right)(1+o(1))n^2.\]

We have already mentioned the trivial inequalities $W_d(n),D_d(n)\leq U_d(n)$.
It remains to show that
\[U_d(n)\leq \frac12\left(1-\frac{1}{a(d)}\right)(1+o(1))n^2.\]
Suppose this is false.
Then, by the Erd\H{o}s-Stone theorem, for arbitrarily large $m\in\N$ there exists a family $\{A_i\colon i=1,\dots,a(d)+1\}$ with each $A_i$ a set of $m$ points in $\R^d$, and a strictly convex norm on $\R^d$, such that the distance between any two points from different $A_i$ is $1$ in this norm.
By the triangle inequality, the diameter of each $A_i$ is at most $2$.
By Lemma~\ref{covering} below, each $A_i$ has a subset $A_i'$ of at least $c_d m$ points and of diameter less than $1$, for some $c_d>0$ depending only on $d$.
Thus the distance between two points in different $A_i'$ is the diameter of the set $\bigcup_i A_i'$.
It follows, again from the definition of strict antipodality, that $\{A_i'\colon i=1,\dots,a(d)+1\}$ is a strictly antipodal family of more than $a(d)$ sets.
Since the size of each $A_i'$ is arbitrarily large, we obtain a contradiction.
\end{proof}

\begin{lemma}\label{covering}
Let $A$ be a set of $m$ points of diameter $1$ in a $d$-dimensional normed space.
Then for any $\lambda\in(0,1)$, $A$ has a subset $A'$ of diameter at most $\lambda$ and with \[\card{A'}\geq \frac{\card{A}}{(1+\lambda)^{d+O(\log d)}}.\]
\end{lemma}
\begin{proof}
According to a result of Rogers and Zong \cite{RZ}, if $N$ is the smallest number of translates of a convex body $H$ that cover a convex body $K$, then
\[ N\leq \frac{\vol{K-H}}{\vol{H}}(d\log d +d\log\log d+5d). \]
Applying this to $K=\conv{A}$ and $H=-\lambda K$, we obtain that there are at most $(1+\lambda)^d O(d\log d)$ translates of $-\lambda \conv{A}$ (each of diameter $\lambda$) that cover $\conv{A}$.
By the pigeon-hole principle, one of the translates contains at least $\frac{\card{A}}{(1+\lambda)^d O(d\log d)}$ points of $A$.
\end{proof}

\end{document}